\documentclass[a4paper, 11pt]{amsart}
\usepackage{amsfonts, amsthm, amssymb, amsmath}
\usepackage{mathrsfs,array}
\usepackage{xy}
\usepackage{hyperref}
\usepackage{verbatim}
\usepackage{cancel}
\usepackage[latin1]{inputenc}
\usepackage{tikz-cd}
\usepackage{bbm}
\usepackage{stmaryrd}

\input xy
\xyoption{all}

\newtheorem{theo}{Theorem}[section]
\newtheorem{prop}[theo]{Proposition}
\newtheorem{defi}[theo]{Definition}
\newtheorem{lemm}[theo]{Lemma}
\newtheorem{coro}[theo]{Corollary}

\newtheorem{conj}[theo]{Conjecture}

\theoremstyle{definition}

\newtheorem{rema}[theo]{Remark}

\newtheorem{conv}[theo]{Convention}

\newtheorem{setup}[theo]{Setup}

\newcommand{\wh}{\widehat}

\newcommand{\mc}{\mathcal}
\newcommand{\mf}{\mathfrak}

\newcommand{\sub}{\subseteq}

\newcommand{\Qp}{\mathbb{Q}_{p}}

\newcommand{\Q}{\mathbb{Q}}
\newcommand{\Z}{\mathbb{Z}}

\newcommand{\bD}{\mathbb{D}}
\newcommand{\ol}{\overline}

\newcommand{\bu}{\bullet}

\newcommand{\oo}{\mathcal{O}}

\newcommand{\Bun}{\mathrm{Bun}}
\newcommand{\lis}{\mathrm{lis}}

\newcommand{\D}{\mathcal{D}}
\newcommand{\ZZ}{\mathcal{Z}}
\newcommand{\spec}{\mathrm{spec}}
\newcommand{\geom}{\mathrm{geom}}
\newcommand{\Hecke}{\mathrm{Hecke}}
\newcommand{\basic}{\mathrm{basic}}
\newcommand{\m}{\mathfrak{m}}
\newcommand{\qlb}{\overline{\mathbb{Q}}_{\ell}}
\newcommand{\X}{\mathfrak{X}}
\newcommand{\Y}{\mathfrak{Y}}

\DeclareMathOperator{\Hom}{Hom}

\DeclareMathOperator{\End}{End}

\DeclareMathOperator{\GL}{GL}
\DeclareMathOperator{\PGL}{PGL}
\DeclareMathOperator{\SL}{SL}
\DeclareMathOperator{\GSp}{GSp}

\DeclareMathOperator{\Rep}{Rep}

\DeclareMathOperator{\Sp}{Sp}
\DeclareMathOperator{\SO}{SO}

\DeclareMathOperator{\Lie}{Lie}

\DeclareMathOperator{\Perf}{Perf}

\DeclareMathOperator{\Sht}{Sht}
\DeclareMathOperator{\Act}{Act}
\DeclareMathOperator{\Mant}{Mant}
\DeclareMathOperator{\Irr}{Irr}
\DeclareMathOperator{\QCoh}{QCoh}
\DeclareMathOperator{\Mod}{Mod}

\DeclareMathOperator{\RHom}{RHom}

\title{A note on the cohomology of moduli spaces of local shtukas}

\author{David Hansen} 
\address{David Hansen\\Department of Mathematics, National University of Singapore, 10 Lower Kent Ridge Road,
Singapore 119076}
\email{dhansen@nus.edu.sg}
\urladdr{http://www.davidrenshawhansen.net}

\author{Christian Johansson}
\address{Christian Johansson\\Department of Mathematical Sciences, Chalmers University of Technology and the University of Gothenburg, 412 96 Gothenburg, Sweden}
\email{chrjohv@chalmers.se}
\urladdr{http://www.math.chalmers.se/~chrjohv/}

\begin{document}

\begin{abstract}
We study localized versions of spectral action of Fargues--Scholze, using methods from higher algebra. As our main motivation and application, we deduce a formula for the cohomology of moduli spaces of local shtukas under certain genericity assumptions, and discuss its relation with the Kottwitz conjecture. 
\end{abstract}

\maketitle

\section{Introduction}

The purpose of this note is to explicate a formula for the cohomology of moduli spaces of local shtukas that can be derived from the recent work of Fargues--Scholze \cite{fargues-scholze}, and to note some consequences. Let $p \neq \ell$ be distinct primes. The theory of local Shimura varieties, which began with examples such as the Lubin--Tate and Drinfeld towers, and continued with the moduli spaces of $p$-divisible groups of Rapoport--Zink \cite{rapoport-zink}, has reached a new level of generality with Scholze's definition of moduli spaces of local shtukas \cite{berkeleylectures}. For any connected reductive group $G/\Qp$, any element $b$ in Kottwitz's set $B(G)$, and any finite collection $\mu_{\bu} = (\mu_i)_{i\in I}$ of conjugacy classes of cocharacters $\mu_i$ of $G$, Scholze defines a tower 
\[
(\Sht_{(G,b,\mu_\bu),K})_K
\]
of diamonds, with inverse limit $\Sht_{(G,b,\mu_\bu)}$, where $K \sub G(\Qp)$ runs through the compact open subgroups. When $\mu_\bu = \mu$ is a single minuscule cocharacter, the $\Sht_{(G,b,\mu),K}$ are smooth rigid spaces referred to as local Shimura varieties. The space $\Sht_{(G,b,\mu_\bu)}$ carries commuting actions of $G(\Qp)$ and $G_b(\Qp)$, where $G_b$ is the inner form of a Levi subgroup of the quasisplit form of $G$ canonically attached to the datum $(G,b)$. For any irreducible admissible $G_b(\Qp)$-representation $\rho$ over $\qlb$, one may define the  ``$\rho$-isotypic part of the $\qlb$-intersection cohomology of $\Sht_{(G,b,\mu_\bu)}$'', which we will denote by  
\[
R\Gamma(G,b,\mu_\bu)[\rho].
\]
We note that this definition naturally incorporates a shift making $0$ the `middle degree'. It carries commuting actions of $G(\Qp)$ and $W_{E_\bu} := \prod_{i\in I} W_{E_i}$ (where $W_{E_i}$ is the Weil group of the reflex field $E_i$ of $\mu_i$) and is a bounded complex of finite length admissible $G(\Qp)$-representations. Roughly speaking, the local Langlands conjecture associates an $L$-parameter $\phi = \phi_\rho : W_{\Qp} \to (\wh{G} \rtimes W_{\Qp})(\qlb)$ with $\rho$, as well as an $L$-packet $\Pi_\phi(G)$ of irreducible admissible representations of $G(\Qp)$. Furthermore, there should be a relation between $\Irr(S_\phi,\chi_b)$ and $\Pi_\phi(G)$. Here $S_\phi$ is the centralizer in $\wh{G}$ of the image of $\phi$, and $\Irr(S_\phi,\chi_b)$ denotes the set of irreducible representations of $S_\phi$ on which $Z(\wh{G})^{W_{\Qp}} \sub S_\phi$ acts by a certain character $\chi_b$ determined by $b$. Let $V$ be the dual of the irreducible representation of $\prod_{i\in I} \wh{G} \rtimes W_{E_i}$ with extreme weight $\mu_\bu$. The underlying vector space of $V$ also carries an action of $W_{E_\bu}$ coming from $\phi$, and this action makes it into a $S_\phi \times W_{E_\bu}$-representation that we will denote by $V_\phi$. Define
\[
\Mant_{G,b,\mu_\bu}(\rho) := \sum_{n} (-1)^n H^n(R\Gamma(G,b,\mu_\bu)[\rho])
\]
in the Grothendieck group of $G(\Qp) \times W_{E_\bu}$-representations. Write $\Rep(S_\phi,\chi_b)$ for the set of finite-dimensional representations of $S_\phi$ on which $Z(\wh{G})^{W_{\Qp}}$ acts by $\chi_b$. The following is then the natural generalization of the Kottwitz conjecture (see e.g. \cite[\S 7.1]{rapoport-viehmann}, \cite[Conjecture 1.0.1]{hkw}), together with a folklore vanishing conjecture.

\begin{conj}\label{kottwitz conjecture} Assume that $b$ is basic and $\phi = \phi_\rho$ is elliptic, i.e. $\phi$ is semisimple and $S_\phi/Z(\wh{G})^{W_{\Qp}}$ is finite. In this case a construction assuming the local Langlands correspondence gives a map $\Pi_\phi(G) \to \Rep(S_\phi,\chi_b)$, which we denote by $\pi \mapsto \delta_{\pi,\rho}$.

\medskip

\begin{enumerate}
\item (Kottwitz conjecture) We have $\Mant_{G,b,\mu_\bu}(\rho) = \sum_{\pi \in \Pi_\phi(G)} \pi \boxtimes \Hom_{S_\phi}(\delta_{\pi,\rho},V_{\phi})$.

\medskip

\item (Vanishing conjecture) $H^n(R\Gamma(G,b,\mu_\bu)[\rho]) = 0$ for $n \neq 0$.
\end{enumerate}
\end{conj}

We will refer to the conjunction of parts (1) and (2) as the \emph{strong Kottwitz conjecture}. When $b$ is not basic, an extension of the Harris--Viehmann conjecture \cite[Conjecture 8.4]{rapoport-viehmann} gives a formula for $\Mant_{G,b,\mu_\bu}(\rho)$ in terms of the basic case for smaller reductive groups; we will not elaborate on this further. Our main goal is a formula for $R\Gamma(G,b,\mu_\bu)[\rho]$. We recall that \cite{fargues-scholze} associates a semisimple $L$-parameter with any irreducibe admissible representation of a connected reductive group; we call this the Fargues--Scholze parameter. It is expected to be semisimplication of the $L$-parameter appearing in the local Langlands conjecture. Our main result is then the following:

\begin{theo}\label{main intro}
Assume that the Fargues--Scholze parameter $\varphi$ attached to $\rho$ is generous. Then
\[
R\Gamma(G,b,\mu_\bu)[\rho] \cong \bigoplus_{\delta \in \Irr(S_\phi,\chi_b)} C_\delta \boxtimes \Hom_{S_\varphi}(\delta,V_{\varphi}),
\]
in the derived category of $G(\Qp) \times W_{E_\bu}$-representations. If $\varphi$ is elliptic, then $C_\delta$ is a nonzero split bounded complex of finite direct sums of supercuspidal representations of $G(\Qp)$ with Fargues--Scholze parameter $\varphi$.
\end{theo}

The precise version, which notably includes a formula for each $C_\delta$ in terms of $\rho$ and $\delta$, is given in Theorem \ref{main formula} and Corollary \ref{main elliptic} (and works in positive characteristic as well). A parameter $\varphi$ is called generous if it is semisimple, if no other $L$-parameter has semisimplification $\varphi$, and if $\varphi$ satisfies an additional technical moduli-theoretic condition (see Definition \ref{def: generous}). Generous parameters are generic on the coarse moduli space and includes the elliptic parameters. Thus, under the expectation that $\varphi = \phi$, the action of $W_{E_\bu}$ on $R\Gamma(G,b,\mu_\bu)[\rho]$ is essentially as predicted from the Kottwitz conjecture. We note that our formula may be seen as a (more general) local analogue of \cite[Prop. 1.2]{lafforgue-zhu}. 

\medskip

The proof of Theorem \ref{main intro} is given in sections \ref{sec: action} and \ref{sec: formula}. The main idea is to apply the machinery of higher algebra to the spectral action of Fargues--Scholze, to obtain a version of the spectral action which only sees one $L$-parameter at a time. The general version of this idea is described in \S \ref{sec: action}, and we believe that this should be a useful tool in the study of the spectral action in general. In \S \ref{sec: formula} we use it to derive Theorem \ref{main intro}. The key point that we wish to make, and which is needed to carry out the proof, is that even though the machinery that we use (monoidal $\infty$-categories and their modules) is highly abstract, it allows you to make computations. This would fail if we tried to work with triangulated categories instead of their $\infty$-categorical enhancements.

\medskip

Section \ref{sec: consequences} then gives some applications of Theorem \ref{main intro} to both parts of Conjecture \ref{kottwitz conjecture}; we highlighting one such application here. When disregarding the $W_{E_\bu}$-action, Conjecture \ref{kottwitz conjecture} was recently proven in \cite{hkw} under the assumption of a precise form of the local Langlands correspondence. Combining this with Theorem \ref{main intro}, one gets the following result.

\begin{theo}\label{intro kottwitz}
Assume the refined local Langlands correspondence \cite[Conjecture G]{kaletha}, and let $\phi$ be the $L$-parameter attached to $\rho$. Assume further that the Fargues--Scholze parameter $\varphi$ attached to $\rho$ is elliptic, that all $\delta \in \Irr(S_\varphi,\chi_b)$ are one-dimensional, and that all representations in $\Pi_\phi(G)$ are supercuspidal. Then there exists a surjection $\delta \mapsto \pi_\delta$ from $\Irr(S_\varphi,\chi_b)$ to $\Pi_\phi(G)$ such that
\[
\Mant_{G,b,\mu_\bu}(\rho) = \sum_{\delta \in \Irr(S_\varphi,\chi_b)} \pi_\delta \boxtimes \Hom_{S_\varphi}(\delta,V_\varphi)
\]
in the Grothendieck group of $G(\Qp) \times W_{E_\bu}$-representations.
\end{theo}

\subsection*{Acknowledgements}

We thank the anonymous referee for many comments which helped improve the paper. We also thank Pol van Hoften for helpful discussions on a previous version of this paper. C.J. was supported by Vetenskapsr\r{a}det Grant 2020-05016, Geometric structures in the p-adic Langlands program. D.H. was supported by a startup grant through the National University of Singapore.

\section{Spectral action with supports}\label{sec: action}

In this section we derive a version of the spectral action, taking supports in the coarse moduli space of $L$-parameters into account. A very special case of this is considered in \cite{fargues-scholze}, when the support is a connected component of the coarse moduli space. The general statement turns out to be rather formal, using the machinery of higher algebra. We will use the same notation as in \cite{fargues-scholze} as much as possible, except that we will work over $\qlb$ as opposed to the more general $\Z_\ell$-algebras $\Lambda$ considered in \cite{fargues-scholze}\footnote{Working over $\qlb$, as opposed to over a general field extension of $\Q_\ell(\sqrt{q})$, is just a matter of convenience. We would hope that much of our picture carries over to all cases considered in \cite{fargues-scholze}, but there are additional technical challenges in mixed characteristic.}. 

\medskip

We start with a quick recap of some of the main players. In what follows, $\ell \neq p$ will be two distinct primes and $G$ be a connected reductive group over a local field $E$ of residue characteristic $p$. The size of the residue field of $E$ will be denoted by $q$. The dual group of $G$ over $\qlb$ will be denoted by $\wh{G}$. $\wh{G}$ carries an action of the Weil group $W_E$, which factors through a finite quotient $Q$, and one can form the semidirect product $\wh{G} \rtimes Q$. In \cite[\S III]{fargues-scholze}, the Artin v-stack $\Bun_G$ of $G$-bundles on the Fargues--Fontaine curve is defined. Its underlying topological space $|\Bun_G|$ is naturally identified with Kottwitz's set $B(G)$, and for any $b\in B(G)$ we have an immersion
\[
i^b :  \Bun_G^b \to \Bun_G.
\]
The main player on the geometric side of the geometrization of the local Langlands correspondence is the stable $\infty$-category $\D_{\lis}(\Bun_G,\qlb)$, defined in \cite[\S VII.7]{fargues-scholze}, and its counterparts $\D_{\lis}(\Bun_G^b,\qlb)$ on the strata $\Bun_G^b$, which are equivalent to the derived categories $\D(G_b(E),\qlb)$ of smooth representations of $G_b(E)$ over $\qlb$. 

\begin{conv}
Since we will only work over $\qlb$, we drop it from the notation for the objects on the geometric side, writing $\D_{\lis}(\Bun_G) := \D_{\lis}(\Bun_G,\qlb)$, $\D_{\lis}(\Bun^b_G) := \D_{\lis}(\Bun^b_G,\qlb)$, $\D(G_b(E)) :=\D(G_b(E),\qlb)$, etc.
\end{conv}

For any stable $\infty$-category $\D(\dots)$, its homotopy category will be denoted by $D(\dots)$, and in any ($\infty$-) category $\mc{C}$, we will let $\mc{C}^\omega$ denote the full subcategory of compact objects. $\D_{\lis}(\Bun_G)$ carries the action of Hecke operators \cite[Thm. IX.0.1]{fargues-scholze}: For every finite set $I$ and any algebraic representation $V$ of $(\wh{G} \rtimes Q)^I$ over $\qlb$, there is an exact functor
\[
T_V : \D_{\lis}(\Bun_G) \to \D_{\lis}(\Bun_G)^{BW_E^I}
\]
which preserves compact objects, where $\D_{\lis}(\Bun_G)^{BW_E^I}$ denotes the category of $W_E^I$-equivariant objects in $\D_{\lis}(\Bun_G)$ (see \cite[\S IX.1]{fargues-scholze} for precise definitions). The $T_V$ fit together into exact $\Rep(Q^I)$-linear monoidal functors
\[
\Rep((\wh{G} \rtimes Q)^I) \to \End_{\qlb}(\D_{\lis}(\Bun_G)^\omega)^{BW_E^I},
\]
which are functorial as $I$ varies. Here and below, whenever $\mc{C}$ is a stable $\qlb$-linear $\infty$-category, we let $\End_{\qlb}(\mc{C})$ denote the $\infty$-category of \emph{exact} $\qlb$-linear endofunctors of $\mc{C}$. Moreover, $\End^L_{\qlb}(\mc{C})\sub \End_{\qlb}(\mc{C})$ will denote the full subcategory of colimit-preserving endofunctors (note that colimit-preserving functors are exact, by \cite[Prop. 1.1.4.1]{lurie-ha}).

\medskip

On the spectral side, we have the stack of $L$-parameters $Z^1(W_E,\wh{G})/\wh{G}$ over $\qlb$, defined in \cite[\S VIII]{fargues-scholze}. Since it will figure extensively in this paper, we simplify the notation by defining
\[
\X := Z^1(W_E,\wh{G})/\wh{G}.
\]
We also let $X^\Box := Z^1(W_E,\wh{G})$ be the representation variety and let $X := Z^1(W_E,\wh{G}) \sslash \wh{G}$ be its GIT quotient (the character variety), which is a good moduli space for $\X$. The spectral Bernstein center $\ZZ^{\spec}(G) := \oo(X^\Box)^{\wh{G}}$ is then the ring of global functions on $\X$, and we let $\Perf(\X)$ denote the stable $\infty$-category of perfect complexes on $\X$. By \cite[Thm. X.0.1]{fargues-scholze}, the action of the Hecke operators on $\D_{\lis}(\Bun_G)$ is equivalent to an exact $\qlb$-linear monoidal functor
\[
\Perf(\X) \to \End_{\qlb}(\D_{\lis}(\Bun_G)^\omega),
\]
called the spectral action. 

\begin{rema}\label{reduction to the affine case}
The fact that $\X$ is not quasicompact is sometimes a nuisance. In particular, the spectral action is constructed by writing $\X$ as the union of quasicompact closed and open substacks $(\X^P)_P$ parametrized by the open subgroups $P$ of the wild inertia of $W_E$ which act trivially on $\wh{G}$. Write $X^P \sub X$ for the corresponding affine open and closed subset. For each $P$ there is an associated full subcategory $\D_{\lis}^P(\Bun_G)^\omega \sub \D_{\lis}(\Bun_G)^\omega$. These are direct summands and moreover $\D_{\lis}(\Bun_G)^\omega$ is the union of the $\D_{\lis}^P(\Bun_G)^\omega$. We let $\D_{\lis}^P(\Bun_G) \sub \D_{\lis}(\Bun_G)$ denote the full subcategory generated by $\D_{\lis}^P(\Bun_G)^\omega$ under filtered colimits. Then one has $\D_{\lis}(\Bun_G) = \varprojlim_P \D_{\lis}^P(\Bun_G)$. The spectral action is then constructed as a system of compatible actions
\[
\Perf(\X) \to \Perf(\X^P) \to \End_{\qlb}(\D_{\lis}^P(\Bun_G,)^\omega).
\]
In particular, they extend uniquely to colimit-preserving actions
\[
\QCoh(\X) \to \QCoh(\X^P) \to \End^L_{\qlb}(\D_{\lis}^P(\Bun_G)),
\]
and this induces an action $\QCoh(\X) \to \End^L_{\qlb}(\D_{\lis}(\Bun_G))$. See \cite[\S IX.5]{fargues-scholze} for precise definitions of the objects above. If $V \in \QCoh(\X)$, we will write $V \ast -$ or $\Act_V(-)$ for the corresponding endofunctor on $\D_{\lis}(\Bun_G)$, depending on which notation fits best with the situation at hand.
\end{rema}

The spectral action makes $\D_{\lis}(\Bun_G)$ into a module for $\QCoh(\X)$ in the sense of higher algebra (we refer to \cite{lurie-ha} for the notions of higher algebra). This module structure will allow us to apply the constructions of higher algebra to $\D_{\lis}(\Bun_G)$. Before doing so, we recall a few more properties. First, the Hecke action is recovered from the spectral action as the composition
\[
\Rep((\wh{G} \rtimes Q)^I) \to \Perf(\X)^{BW_E^I} \to \End_{\qlb}(\D_{\lis}(\Bun_G)^\omega)^{BW_E^I},
\]
where the second functor is induced from the spectral action. The first functor sends $V \in \Rep((\wh{G} \rtimes Q)^I)$ to the vector bundle $V \otimes \oo_{X^\Box}$ on $X^\Box$, with $W_E^I$-action given by 
\[
W_E^I \to (\wh{G} \rtimes Q)^I(\oo_{X^\Box}) \to \GL(V \otimes \oo_{X^\Box}),
\]
where the first map is the universal homomorphism in each factor, and the $\wh{G}$-descent datum comes from the diagonal embedding $\wh{G} \to (\wh{G} \rtimes Q)^I$. By looking at centres, the spectral action induces a ring homomorphism
\[
\ZZ^{\spec}(G) \to \ZZ^{\geom}_{\Hecke}(G) \sub \ZZ^{\geom}(G)
\]
to the geometric Bernstein center $\ZZ^{\geom}(G)$ of $G$, landing inside the subring $\ZZ^{\geom}_{\Hecke}(G)$ of endomorphisms equivariant for the Hecke action \cite[Def. IX.0.2]{fargues-scholze}. For any object $A \in \D_{\lis}(\Bun_G)$ we have a ring homomorphism $F_A : \ZZ^{\spec}(G) \to \End_{D_{\lis}(\Bun_G)}(A)$. When $A$ is Schur (i.e. $\End_{D_{\lis}(\Bun_G)}(A)=\qlb$), the kernel of $F_A$ corresponds to a $\qlb$-point of $X$, which we refer to as the Fargues--Scholze parameter of $A$. By \cite[Theorem 5.2.1]{zou2022categorical}, this construction is compatible with that given in \cite[\S IX.4]{fargues-scholze}. For $b\in B(G)$ and $\rho$ an irreducible admissible representation of $G_b(E)$, we will write ``the Fargues--Scholze parameter of $\rho$'' to mean the Fargues--Scholze parameter of $i^b_\ast \rho$\footnote{We note that using $i_\natural^b \rho$ instead of $i^b_\ast \rho$, as is done in \cite[IX.7]{fargues-scholze}, produces the same result (e.g. by the argument in the proof of Proposition \ref{linearity and induced maps}). Moreover, the Fargues--Scholze parameter of $i^b_\ast\rho$ is the composite of the ``correct'' Fargues--Scholze parameter of $\rho$ (defined using $i^1_! \rho$ in $\D_{\lis}(\Bun_{G_b})$ as in \cite[Def. IX.7.1]{fargues-scholze}) and the twisted $L$-embedding $\wh{G}_b \rtimes Q \to \wh{G} \rtimes Q$, by \cite[Thm IX.7.2]{fargues-scholze}.}. 

\medskip

Now consider a derived stack $\Y$ with a map $f : \mf{Y} \to \X$, which induces a symmetric monoidal functor $f^\ast : \QCoh(\X) \to \QCoh(\Y)$. Using the spectral action and $f^\ast$, we can then form the tensor product
\[
\QCoh(\Y) \otimes_{\QCoh(\X)} \D_{\lis}(\Bun_G)
\]
in higher algebra, which is a $\QCoh(\Y)$-module. The basic idea of this paper is that there should be plenty of interesting objects in $\D_{\lis}(\Bun_G)$ which have $\QCoh(\Y)$-structures (for suitable $\Y$), and that this can be used for concrete computations. 

\medskip

To keep the definitions of this paper close to its theorems, we make our definitions in a restricted setting, where things are more concrete. Consider the character variety $X$ from above. Let $Y$ be a closed subscheme of $X$, and consider its \emph{derived} pullback $\Y$ to $\X$. The structure sheaf $\oo_{\Y}$ is a commutative algebra object of $\QCoh(\X)$. Similarly, the structure sheaf $\oo_Y$ of $Y$ is a commutative algebra object of $\QCoh(X)$, and pullback gives a symmetric monoidal functor $\QCoh(X) \to \QCoh(\X)$ sending $\oo_Y$ to $\oo_{\Y}$. In particular, we may and will also view $\D_{\lis}(\Bun_G)$ as module for $\QCoh(X)$.

\begin{defi}\label{definition of categories}
Let $ Y \sub X$ be a closed subscheme and let $\Y$ be its derived pullback to $\X$. 
\begin{enumerate}
\item We define $\QCoh^{Y}(\X)$ to be the $\infty$-category $
\Mod_{\oo_{\Y}}(\QCoh(\X))$ of $\oo_{\Y}$-module objects in $\QCoh(\X)$. Equivalently, $\QCoh^{Y}(\X)$ is $\Mod_{\oo_Y}(\QCoh(\X))$, regarding $\QCoh(\X)$ as a module for $\QCoh(X)$. We also set $\QCoh^{Y}(X) := \Mod_{\oo_Y}(\QCoh(X))$.

\medskip

\item We define $\D_{\lis}^Y(\Bun_G)$ to be the $\infty$-category $\Mod_{\oo_{\Y}}(\D_{\lis}(\Bun_G))$ of $\oo_{\Y}$-module objects in $\D_{\lis}(\Bun_G)$. Equivalently, $\D_{\lis}^Y(\Bun_G)$ is $\Mod_{\oo_Y}(\D_{\lis}(\Bun_G))$, regarding $\D_{\lis}(\Bun_G)$ as a module for $\QCoh(X)$. 

\end{enumerate}

\medskip

When $Y=\{\varphi\}$ is a closed point, we will simply write $\QCoh^{\varphi}(\X)$ and $\D_{\lis}^\varphi(\Bun_G)$.
\end{defi}

\begin{rema}
We note that our $\D_{\lis}^\varphi(\Bun_G)$ differs from the category $\D_{\lis}(\Bun_G)_\varphi$ defined by the first author in \cite[Definition A.1]{hamann-lee}. An analogy is the following: If $A$ is a commutative ring and $\m \sub A$ is a maximal ideal, then $\D(A/\m)$ is to $\D(A)$ as our $\D_{\lis}^\varphi(\Bun_G)$ is to $\D_{\lis}(\Bun_G)$, whereas $\D(\wh{A}_{\m})$ is to $\D(A)$ as $\D_{\lis}(\Bun_G)_\varphi$ is to $\D_{\lis}(\Bun_G)$. Here $\wh{A}_{\m}$ is the complete local ring of $A$ at $\m$, and $\D(\dots)$ denotes the unbounded derived category of a ring (viewed as a stable $\infty$-category). We will not use the categories $\D_{\lis}(\Bun_G)_\varphi$ in this paper.  
\end{rema}
 
We refer to \cite[\S 4]{lurie-ha} for the theory of algebra objects and their modules in higher algebra. By \cite[Prop. 7.1.1.4]{lurie-ha}, all $\infty$-categories defined in Definition \ref{definition of categories} are stable $\infty$-categories (using exactness of the spectral action). Moreover, $\QCoh^{Y}(\X)$ is a symmetric monoidal $\infty$-category and $\D_{\lis}^Y(\Bun_G)$ is a module for it. Both may also viewed as modules for $\QCoh^{Y}(X)$ via $\QCoh^Y(X) \to \QCoh^Y(\X)$. We then have the following well known fact, which is a standard application of Lurie's Barr--Beck Theorem \cite[Thm. 4.7.3.5]{lurie-ha}.

\begin{prop}
$\QCoh^Y(X)$ and $\QCoh^{Y}(\X)$ are equivalent to $\QCoh(Y)$ and $\QCoh(\Y)$, respectively.
\end{prop}

As a consequence, we get the following. 

\begin{coro}
Let $Y$ be a closed subvariety of $X$. Then the spectral action makes $\D_{\lis}^Y(\Bun_G)$ into a module for $\QCoh(\Y)$.
\end{coro}

For computations, we will also need to consider the $\D_{\lis}(\Bun^b_G)$. Associated with $i^b : \Bun_G^b \to \Bun_G$ we have an adjoint pair of functors $(i^{b\ast},i_\ast^b)$, defined in \cite[\S VII.6]{fargues-scholze}. Since $\QCoh(X)$ is generated by $\oo_X$ under cones, shifts, retracts and filtered colimits, any idempotent complete, cocomplete stable subcategory of $\D_{\lis}(\Bun_G)$ will be preserved by the action of $\QCoh(X)$. We can therefore define a $\QCoh(X)$-action on $\D_{\lis}(\Bun_G^b)$ by declaring that
\[
i^b_{\ast} : \D_{\lis}(\Bun_G^b) \to \D_{\lis}(\Bun_G)
\]
is $\QCoh(X)$-linear.

\begin{prop}\label{linearity and induced maps}
The functor $i^{b\ast} : \D_{\lis}(\Bun_G) \to \D_{\lis}(\Bun_G^b)$ is $\QCoh(X)$-linear with respect to the action above. In particular, for any $Y\sub X$ as above, $i^{b\ast}$ and $i^b_\ast$ induce functors $\D_{\lis}^Y(\Bun_G)\to \D_{\lis}^Y(\Bun^b_G)$ and $\D_{\lis}^Y(\Bun^b_G)\to \D_{\lis}^Y(\Bun_G)$, which we will also denote by $i^{b\ast}$ and $i^b_\ast$, respectively.
\end{prop}

\begin{proof}
The second part follows from the first part (and the definition, in the case of $i^b_\ast$). The first part is essentially follows from \cite[Cor. 6.2.4]{gaitsgory-notesondgcats}, except that $\QCoh(X)$ is not rigid since $X$ is not quasicompact. However, the categories $\QCoh(X^P)$ are rigid, so using Remark \ref{reduction to the affine case} one reduces to this case. We omit the details. 
\end{proof}

For our applications, we need a criterion for objects in $\D_{\lis}(\Bun_G)$ to have an $\oo_{Y}$-structure, phrased in terms of Fargues--Scholze parameters. For this, we will use the following lemma.

\begin{prop}\label{actions}
Let $k$ be a field and let $A$ and $B$ be $k$-algebras, with $A$ commutative. Let $F : \D(A) \to \End_k^L(\D(B))$ be a monoidal $k$-linear functor which commutes with colimits. Then: 

\begin{enumerate}
\item $F$ induces a ring homomorphism $f : A \to Z(B)$, where $Z(B)$ denotes the centre of $B$, and $F$ is given by $F(M) = ( N \mapsto M \otimes_{A,f}^L N)$, for $M \in \D(A)$ and $N \in \D(B)$.

\smallskip

\item Let $I \sub A$ be an ideal and let $M \in \D(B)$. Assume that $M$ has homology in a single degree. If the map $A \to \End_{D(B)}(M)$ factors through $A/I$, then $M$ has a canonical structure of an $A/I$-module, and hence lies in $\D(A/I \otimes_A^L B) = \Mod_{A/I}(\D(B))$. 

\end{enumerate}
\end{prop}

\begin{proof}
We start with part (1). For any $k$-algebras $R$ and $S$, the $\infty$-category of colimit-preserving functors $\D(R) \to \D(S)$ is equivalent to $\D(S \otimes_k R^{op})$, or equivalently to the derived $\infty$-category $\D(S,R)$ of $(S,R)$-bimodules for which  the $k$-module structures coming from $S$ and $R$ agree. Here $X \in \D(S,R)$ corresponds to the functor $M \mapsto X\otimes_R^L M$, and for a given $F : \D(R) \to \D(S)$ the corresponding $X$ is $F(R)$, with the left $R^{op}$-module structure coming from the map
\[
R^{op}=\End_{\D(R)}(R) \to \End_{\D(S)}(F(R))
\] 
induced by $F$. When $R=S$, the monoidal structure on $\End^L_k(\D(R))$ by composition corresponds to the tensor product (over $R$) on $\D(R,R)$. Now consider the $F$ in the statement of the proposition. Using the above, we may think of it as a monoidal functor $\D(A) \to \D(B,B)$, and hence $F(A)$ is equivalent to $B$, the unit. Applying the above remarks again, now thinking of $\D(B,B)$ as $\D(B\otimes_k B^{op})$, $F$ itself is given by $M \mapsto B \otimes_A^L M$, with the $A$-module structure on the $(B,B)$-bimodule $B$ being given by the homomorphism
\[
A=\End_{\D(A)}(A) \to \End_{\D(B,B)}(B).
\]
Since $B$ is concentrated in degree $0$, $\End_{\D(B,B)}(B)$ is coconnective. Since $A$ is also concentrated in degree $0$, this means that the map $A \to \End_{\D(B,B)}(B)$ factors through the center $Z(B) = \pi_0(\End_{\D(B,B)}(B))$, giving the desired map $f$. Translating back from $\D(B,B)$ to $\End_k^L(\D(B))$ then gives the desired formula. 

\medskip

For part (2), by shifting we may assume that $M$ has homology in degree $0$, and so lies in the heart of $\D(B)$ with respect to the usual t-structure. The assumption then says that $M$ has a canonical $B/IB$-module structure in the usual sense. Restriction along the canonical map $A/I \otimes^L_A B \to B/IB$ then gives a canonical $A/I \otimes_A^L B$-module structure, as desired.
\end{proof}

We then have the following corollary, which gives us elements of $\D_{\lis}^\varphi(\Bun_G)$.

\begin{coro}\label{containmen in D-phi}
Let $\rho$ be an irreducible admissible $G_b(E)$-representation (in the abelian category, not the derived category) with Fargues--Scholze parameter $\varphi$. Then $i_\ast^b\rho$ naturally lives in $\D_{\lis}^\varphi(\Bun_G)$.  
\end{coro}

\begin{proof}
We may choose a sufficiently small compact open subgroup $K \sub G_b(E)$ such that $\rho$ is generated by its $K$-fixed vectors and we have a fully faithful exact colimit-preserving embedding $\D(\mc{H}(G,K)) \sub \D(G_b(E))$, where $\mc{H}(G,K)$ is the usual Hecke algebra of bi-$K$-invariant compactly supported $\qlb$-valued functions on $G_b(E)$. In particular, $\D(\mc{H}(G,K))$ inside $\D(G_b(E)) = \D_{\lis}(\Bun^b_G)$ is $\QCoh(X)$-stable and contains $\rho$, so we may apply Proposition \ref{actions} to deduce that $\rho \in \D_{\lis}^\varphi(\Bun^b_G)$. By Proposition \ref{linearity and induced maps}, we then have $i_\ast^b \rho \in \D_{\lis}^\varphi(\Bun_G)$.
\end{proof}

\section{Cohomology of moduli spaces of local shtukas}\label{sec: formula}

We now apply the constructions of the previous section to the cohomology of moduli space of local shtukas. Let $b\in B(G)$, let $I$ be a finite set and let $\mu_\bu = (\mu_i)_{i\in I}$ be a collection of conjugacy classes of cocharacters of $G$, with reflex fields $E_\bu = (E_i)_{i\in I}$. From the introduction we have the tower $(\Sht_{(G,b,\mu_\bu),K})_K$ of moduli spaces of local shtukas. The ``intersection cohomology complex'' of $\Sht_{(G,b,\mu_\bu),K}$ is the object $f_{K\natural}\mc{S}^\prime_W$ defined before \cite[Prop. IX.3.2]{fargues-scholze}, where $W \in \Rep(\wh{G}^I \rtimes Q_\bu )$ is the representation with highest weight $\mu_\bu = (\mu_i)_{i\in I}$. When $\mu_\bu = \mu$ is a single minuscule cocharacter this is simply the (shifted) compactly supported $\ell$-adic cohomology of the smooth rigid analytic variety $\Sht_{(G,b,\mu),K}$. Let $\rho$ be an irreducible admissible representation of $G_b(E)$, with Fargues--Scholze parameter $\varphi$. Consider the $\rho$-isotypic part 
\[
R\Gamma (G,b,\mu_\bu)[\rho] = \varinjlim_K \RHom_{G_b(E)}(f_{K\natural}\mc{S}^\prime_W,\rho)
\]
of the cohomology of the tower $(\Sht_{(G,b,\mu_\bu),K})_K$. By the proof of \cite[Prop. IX.3.2]{fargues-scholze} and the argument in the proof of \cite[Proposition 6.4.5]{hkw}, we have
\begin{equation}\label{Shtukas and Hecke}
R\Gamma(G,b,\mu_\bu)[\rho] \cong i^{1\ast} T_V i^{b}_\ast \rho,
\end{equation}
where $V=W^\vee \in \Rep(\wh{G}^I \rtimes Q_\bu )$ is the dual of the representation $W$ with highest weight $\mu_\bu = (\mu_i)_{i\in I}$, and this is a bounded complex of finite length $G(E)$-representations. Note that the restriction of $W$ to the diagonally embedded copy of $Z(\hat{G})^{Q}$ in $\Rep(\wh{G}^I \rtimes Q_\bu )$ is necessarily isotypic for some character $\chi$, and the functor $R\Gamma(G,b,\mu_\bu)[-]$ will therefore be identically zero unless $\chi$ equals the Kottwitz point of $b$ under the bijection $\pi_1(G)_{Q} = Z(\hat{G})^{Q}$, by \cite[Lemma 5.3.2]{zou2022categorical}. We will therefore assume this compatibility in the rest of the paper.

\medskip

Now set $\X_\varphi := {\varphi} \times_X \X$ (in the derived sense), viewing $\varphi$ as a closed point of $X$. To understand the cohomology, we have to analyze the expression $i^{1\ast} T_V i^{b}_\ast \rho$ more closely. By Corollary \ref{containmen in D-phi}, $i^b_\ast \rho \in \D_{\lis}^\varphi(\Bun_G)$. The operator $T_V$ then takes us from $\D_{\lis}^\varphi(\Bun_G)$ to $\D_{\lis}^\varphi(\Bun_G)^{BW_E^I}$, and is given by the image of $V$ under the composition 
\[
\Rep(\wh{G}^I \rtimes Q_\bu) \to \QCoh(\X_\varphi)^{BW_{E_\bu}} \to \End_{\qlb}(\D_{\lis}^\varphi(\Bun_G))^{BW_{E_\bu}},
\]
so to understand this better we need to understand the derived stack $\X_\varphi$. A priori, this can be a non-classical stack, and we will have little to say about this case in what follows --- to see the correct structure on $R\Gamma(G,b,\mu_\bu)[\rho]$ in that case one would want a refinement of Corollary \ref{containmen in D-phi}. We will focus on a particular class of parameters $\varphi$ that turn out to satisfy $\X_\varphi \cong [\ast/S_\varphi] $, where $S_\varphi$ is the centralizer of $\varphi$, viewed as a point of $\X$. Recall that semisimple $L$-parameters over $\qlb$ (up to conjugacy) precisely correspond to closed points of $X$. Before defining the classes of parameters that we will be interested in, we state some algebro-geometric conventions that we will use.

\begin{conv}\label{alg geom convs}
When $Y$ is a (classical) scheme of finite type over $\qlb$, or a disjoint union of such schemes, we will often think of it as its set of $\qlb$-points (with the Zariski topology) as in classical algebraic geometry. In particular, ``the locus of points of $Y$ satisfying a particular property'' refers to the subset of $\qlb$-points with that property. Moreover, $Y_{red}$ will denote the nilreduction of $Y$. Whenever $f : Y_1 \to Y_2$ is a finite type morphism of such schemes and $S\sub Y_2$ is a subscheme, we write $f^{-1}(S)$ for the (non-derived) fibre product $S \times_{Y_2} Y_1$. If $S=\{y\}$ is a closed point, we write $f^{-1}(y)$ instead of $f^{-1}(\{y\})$.
\end{conv}

We now return to the context of moduli spaces and stacks of $L$-parameters. In what follows, we write $\pi : X^\Box \to X$ for the map from $X^\Box$ to its GIT quotient $X$. 

\begin{defi}\label{def: generous}
Let $\varphi$ be a semisimple $L$-parameter over $\qlb$.
\begin{enumerate}
\item We say that $\varphi$ is strongly semisimple if $\pi^{-1}(\varphi)_{red}$ consists of a single $\wh{G}$-orbit.

\item We say that $\varphi$ is generous if $\pi^{-1}(\varphi)$ is reduced and consists of a single $\wh{G}$-orbit.
\end{enumerate}
\end{defi}

Concretely, $\varphi$ is strongly semisimple if there are no other $L$-parameters with the same semisimplification as $\varphi$. We have the following lemma:

\begin{lemm}\label{ss implies smooth}
Let $\varphi$ be a strongly semisimple $L$-parameter. Then $\varphi$ defines a smooth point of $X^\Box$ or, equivalently, of $\X$.
\end{lemm}

\begin{proof}
By \cite[\S VIII.2]{fargues-scholze}, the tangent complex of $\X$ at $\varphi$ is $R\Gamma(W_E, \mathrm{ad}(\varphi))[1]$, which is concentrated in $[-1,1]$, and the obstruction group is $H^2(W_E,\mathrm{ad}(\varphi))$. Thus, it suffices to prove that $H^2(W_E,\mathrm{ad}(\varphi))=0$. By \cite[Prop. VIII.2.2]{fargues-scholze}, we have
\[
H^2(W_E,\mathrm{ad}(\varphi)) \cong H^0(W_E,\mathrm{ad}(\varphi)(1))^\ast,
\]
where $-^\ast$ denotes the $\qlb$-linear dual. Since we are in characteristic $0$, we can interpret $L$-parameters as Weil--Deligne parameters (cf. \cite[Def. VIII.2.4, Prop. VIII.2.5]{fargues-scholze}) and, since $\varphi$ is semisimple, $H^0(W_E,\mathrm{ad}(\varphi)(1))$ is precisely the set of $N \in \Lie(\wh{G})$ such that $(\varphi,N)$ is a Weil--Deligne parameter. Any non-zero such $N$ defines an non-semisimple $L$-parameter whose semisimplification is $\varphi$, so such $N$ cannot exist since $\varphi$ is strongly semisimple. It follows that $H^0(W_E,\mathrm{ad}(\varphi)(1))=0$, and $\varphi$ is a smooth point of $\X$.
\end{proof}

For any semisimple $L$-parameter $\varphi$, the underlying classical stack $\X_\varphi^{cl}$ of $\X_\varphi$ is $[\pi^{-1}(\varphi)/\wh{G}$], so by definition $\varphi$ is generous if and only if $\X_\varphi^{cl} \cong [\ast / S_\varphi]$. In fact, we will now show that $\X_{\varphi}$ is classical for generous $\varphi$. Recall our conventions from Convention \ref{alg geom convs}. We isolate a general geometric invariant theory result from the argument. For this, we will need some basic notions and results from affine geometric invariant theory. We refer the reader to \cite[\S 1.2]{mumford-git} and \cite[Thm. 3.5]{newstead-introtomoduli} for the basic properties that we will use. 

\begin{lemm}\label{key GIT}
Let $S$ be an irreducible affine $\qlb$-variety with an action of $\wh{G}$ and let $T = S \sslash \wh{G}$ be its GIT quotient. Write $f : S \to T$ for the quotient map. Let $U\sub S$ denote the union of the $\wh{G}$-orbits which are closed and not contained in the closure of any other $\wh{G}$-orbit. Then there exists an open subset $V\sub T$ such that $U = f^{-1}(V)$, and $f : U \to V$ is a universal geometric quotient. 
\end{lemm}

\begin{proof}
Consider the set $U_1 \sub S$ of points whose orbits have maximal dimension, and denote this maximal dimension by $d^\prime$. By \cite[Lem. 3.7(c)]{newstead-introtomoduli}, $U_1$ is open. Next, let $U_2 \sub U_1$ be the subset of points whose orbits are closed. By \cite[Prop. 3.8]{newstead-introtomoduli}, $U_2 = f^{-1}(V)$ for some open $V \sub T$ , and $U_2 \to V$ is a universal geometric quotient\footnote{For the reader looking at this reference, we note that the notions of orbit spaces and geometric quotients, and their relation, is covered in \cite[p. 31, Def.]{newstead-introtomoduli} and \cite[p. 57, Def.]{newstead-introtomoduli}, and that geometric quotients in characteristic $0$ in affine GIT are universal \cite[Amplification 1.3]{mumford-git}.}. Thus, it suffices to show that $U=U_2$. To show that $U_2 \sub U$, note that any orbit in $U_2$ is closed and cannot be in the closure of any other orbit, since such an orbit would need to have bigger dimension, which is impossible by the definition of $U_2$.

\medskip

It therefore remains to show that $U\sub U_2$. Consider the set $U_3 \sub S$ of $x\in S$ for which the dimension of the local ring $\oo_{f^{-1}(f(x)),x}$ of the fibre $f^{-1}(f(x))$ is minimal. By \cite[\S I.8, Thm. 3 Cor. 3]{mumford-redbook}, $U_3$ is open. Call this minimal dimension $d$. Let $x \in U$. By definition of $d^\prime$, we must have $d^\prime \geq \dim \wh{G}x$. On the other hand, $\wh{G}x = f^{-1}(f(x))$ by the definition of $U$ and it is equidimensional (since it is a single orbit), so we must have $\dim \wh{G}x \geq d$ by the definition of $d$. In particular, we have $d \leq d^\prime$. On the other hand, we have $U_1 \cap U_3 \neq \emptyset$ since $U_1$ and $U_3$ are open and non-empty, and $S$ is irreducible. Choose a point $y \in U_1 \cap U_3$. Then $d = \dim \oo_{f^{-1}(f(y)),y}$ and $d^\prime = \dim \wh{G}y$ by definition. Since $\wh{G}y \sub f^{-1}(f(y))$ and $\wh{G}y$ is equidimensional, we must have $d^\prime \leq d$, and hence we conclude that $d = d^\prime$ and hence that $U \sub U_1$ (since we proved that $d \leq  \dim \wh{G}x \leq d^\prime$ for $x\in U$). Since the orbits in $U$ are closed, we conclude that $U \sub U_2$, as desired, finishing the proof.
\end{proof}

\begin{prop}\label{generous implies flat}
Let $X^\Box_{ss}$ and $X_{ss}$ denote the loci of strongly semisimple $L$-parameters in $X^\Box$ and $X$, respectively. Then $X_{ss}$ is open, $X_{ss}^\Box = \pi^{-1}(X_{ss})$, and $X_{ss}^\Box \to X_{ss}$ is a universal geometric quotient. Moreover, if $\varphi$ is generous, then there is an open neighbourhood $W\sub X$ of $\varphi$ such that $\pi : \pi^{-1}(W) \to W$ is flat.
\end{prop}

\begin{proof}
We start with the first part. As a first reduction step, note that we may prove the proposition component by component on $X$, so let $C \sub X$ be a connected component and set $C^\Box = \pi^{-1}(C)$; this is connected since $\wh{G}$ is. Moreover, if there are no strongly semisimple parameters on $C$ then the assertion is trivial, so assume that $C_{ss} := C \cap X_{ss} \neq \emptyset$. By \cite[Thm. 1.7]{dhkm}, $C$ is irreducible and reduced. However, $C^\Box$ might not be irreducible, preventing us from applying Lemma \ref{key GIT} directly (note that the condition on the orbits included in $U$ in the notation of Lemma \ref{key GIT} precisely corresponds to the definition of strong semisimplicity). Instead, we will show that the Proposition follows by applying Lemma \ref{key GIT} to a particular irreducible component of $C^\Box$. 

\medskip

Write $C^\Box = Y_1 \cup Y_2 \cup \dots \cup Y_r$ as the union of its irreducible components. Each component is $\wh{G}$-invariant, so the sets $\pi(Y_i)$ are closed. Since $\pi(C^\Box)=C$, we may without loss of generality assume that $\pi(Y_1) = C$. By exactness of $\wh{G}$-invariants, it follows that $C$ is the GIT quotient of $Y_1$ as well. If $C_{ss}^\Box := \pi^{-1}(C_{ss}) \sub Y_1$, then we can apply Lemma \ref{key GIT} to conclude the first part of the Proposition. To see that $C_{ss}^\Box \sub Y_1$, note that each orbit in $C_{ss}^\Box$ has to intersect $Y_1$ since the orbits are closed and $C_{ss} \sub \pi(Y_1)$, and then use $\wh{G}$-invariance of $Y_1$\footnote{In fact, since the points in $C_{ss}^\Box$ are smooth by Lemma \ref{ss implies smooth}, they only lie on $Y_1$. This means that $Y_1$ is the only component of $C^\Box$ mapping surjectively onto $C$.}. This finishes the proof of the first part.

\medskip

For the second part, note that we have shown that $C_{ss}^\Box \to C_{ss}$ is equidimensional and universally open (since it is a universal geometric quotient), and we know that $C_{ss}$ is reduced since $C$ is. So assume that $\varphi$ is a generous parameter in $C$. Then $\pi^{-1}(\varphi)$ is reduced by assumption, so by \cite[Cor. (15.2.3)]{ega4-3}, it follows that $C_{ss}^\Box \to C_{ss}$ is flat in a neighbourhood of $\pi^{-1}(\varphi)$. We may then spread out this neighbourhood using the $\wh{G}$-action (since we have a geometric quotient) and take the image under $\pi$ to obtain the desired $W$.
\end{proof}

\begin{coro}
If $\varphi$ is a generous $L$-parameter, then $\X_\varphi = \X_\varphi^{cl} = [\ast/S_\varphi]$. Moreover, for any strongly semisimple $L$-parameter $\varphi$, $\pi$ is flat in a neighbourhood of $\varphi \in X$ if and only if $\varphi \in X$ is smooth. 
\end{coro}

\begin{proof}
We have already proved the first part. For the second part, we may work component by component, and we use the notation from the proof of Proposition \ref{generous implies flat}. There we have showed that $\pi : C^\Box_{ss} \to C_{ss}$ is equidimensional, and we know that $C^\Box_{ss}$ is smooth. That smoothness of $C_{ss}$ at $\varphi$ implies flatness of $\pi$ at $\varphi \in C_{ss}^\Box$ then follows from miracle flatness \cite[Thm. 23.1]{matsumura}, and the converse follows from \cite[Thm. 23.7(i)]{matsumura}.
\end{proof}

\begin{rema}\label{non-generous parameters} Let us make some remarks on the prevalence of generous parameters. Here, we think of $L$-parameters as Weil--Deligne parameters.

\begin{enumerate}
\item In general, there is a Zariski open and dense subset of generous parameters inside $X_{ss}$ by \cite[Tag 0578]{stacks-project} (the full set of generous parameters is locally constructible by \cite[Tag 0579]{stacks-project}, but we do not know if it is open).

\smallskip

\item Let us indicate a case when strongly semisimple and generous parameters can be described rather explicitly. Assume that $G$ is split, and consider the unramified (or principal) component $\X_{ur}$ of $\X$, consisiting of parameters $(\varphi,N)$ for which $\varphi$ is unramified (and hence equivalent to a single element in $\wh{G}$), which is studied in detail in \cite[\S 2]{hellmann-derived}. Write $X^\Box_{ur}$ and $X_{ur}$ for the corresponding components of $X^\Box$ and $X$, respectively. As noted in a footnote in the proof of Proposition \ref{generous implies flat}, $\X_{ur}$ has a unique irreducible component mapping surjectively onto the GIT quotient, and it contains all strongly semisimple parameters. In this case, it is the component where $N=0$, hence isomorphic to $[\wh{G}/\wh{G}]$ (the action being by conjugation). The strongly semisimple parameters correspond to those regular semisimple elements in $\wh{G}$ which do not lie on other components of $\X_{ur}$. By \cite[6.4, Cor.]{steinberg}, we then have
\begin{equation}\label{Steinberg map}
X_{ur}^\Box = \wh{G} \to X_{ur} = \wh{G}\sslash \wh{G} \cong \wh{T}\sslash N(\wh{T}) = \wh{T} \sslash W,
\end{equation}
where $\wh{T}$ is a maximal torus of $\wh{G}$, with normalizer $N(\wh{T})$, and $W=N(\wh{T})/\wh{T}$ is the Weyl group. It is well known in the theory of algebraic groups how this map can be analysed, but let us sketch one way. The morphism $[\wh{T}/N(\wh{T})] \to [\wh{G}/\wh{G}]$ of quotient stacks (induced by $\wh{T}\sub \wh{G}$) is an isomorphism over the regular semisimple locus (on both sides). In particular, in the smooth topology and over the regular semisimple locus, the fibres of the map in equation (\ref{Steinberg map}) are equivalent to the fibres of $\wh{T} \to \wh{T}\sslash W$. The latter are smooth if and only if they are \'etale, i.e. when the centralizer of the regular semisimple element is $\wh{T}$ itself. In particular, this is always the case when the derived group of $\wh{G}$ is simply connected \cite[2.10, Rem.]{steinberg}. In that case, all strongly semisimple unramified parameters are generous. 

\medskip

Let us also note that the smoothness of $X_{ur}$ at a point is well understood in this situation. Write $X_{ur}=\wh{T}\sslash W$ and consider $t \in \wh{T}$, with stabilizer $W_t \sub W$, and image $\ol{t}\in \wh{T} \sslash W$. Then, by the Shephard--Todd theorem \cite[Thm. 7.2.1]{benson}, $\ol{t}$ is a smooth point if and only if the action of $W_t$ of the tangent space $T_t(\wh{T})$ of $\wh{T}$ at $t$ is generated by pseudoreflections\footnote{The analysis of the local geometry of $\wh{T} \sslash W$ at $\ol{t}$ in terms of the action of $W_t$ on $T_t(\wh{T})$ is a standard application of Luna's \'etale slice theorem and related results, cf. \cite[\S 2.2, \S 3.1]{luna}.}. In particular, we can consider the following examples: Let $n\geq 2$ and let $\zeta$ be a primitive $n$-th root of unity. The unramified parameter $\varphi_n$ sending Frobenius to $(1,\zeta,\dots,\zeta^{n-1}) \in \PGL_n(\qlb)$ (with $N=0$) is strongly semisimple, but $S_\varphi$ is cyclic of order $n$, so $\varphi_n$ is not generous. Moreover, a tedious but straightforward calculation shows that it is a smooth point in $X_{ur}$ if and only if $n=2$.

\smallskip

\item The previous part of the remark shows that all unramified strongly semisimple $L$-parameters for $\GL_n$ are generous. In fact, one may bootstrap this to show that \emph{all} strongly semisimple $L$-parameters for $\GL_n$ are generous. This follows from the fact that any component of $\X$ (for $\GL_n$) is isomorphic to a product of unramified components for some $\GL_d$'s (possibly for an extension of $F$). For a reference for this, see the very end of \cite[\S 5.2]{bzchn}.
\end{enumerate}

\end{rema}

We now return to analyzing $R\Gamma(G,b,\mu_\bu)[\rho]$, under the assumption that the Fargues--Scholze parameter $\varphi$ of $\rho$ is generous. Since $\X_\varphi \cong [\ast/S_\varphi]$, we have $\QCoh(\X_\varphi) \cong \Rep(S_\varphi)$ and the action of $V$ on $\D_{\lis}^\varphi(\Bun_G)$ factors as 
\begin{equation}\label{generic factorization}
\Rep(\wh{G}^I \rtimes Q_\bu) \to \Rep(S_\varphi)^{BW_{E_\bu}} \to \End_{\qlb}(\D_{\lis}^\varphi(\Bun_G)^\omega)^{BW_{E_\bu}}.
\end{equation}
By the compatibility between $b$ and $\mu_\bu$ (as recalled in the beginning of this section), $Z(\wh{G})^Q \sub S_\varphi$ acts on $V$ via a character $\chi_b$. Write $\Irr(S_\varphi,\chi_b)$ for the set of irreducible representations of $S_\varphi$ on which $Z(\wh{G})^Q$ acts by $\chi_b$. Also write $V_\varphi$ for the image of $V$ under $\Rep(\wh{G}^I \rtimes Q_\bu) \to \Rep(S_\varphi)^{BW_{E_\bu}}$. We can now put everything together to derive the desired formula.

\begin{theo}\label{main formula}
Let $\rho$ be an irreducible admissible representation of $G_b(E)$ and assume that its Fargues--Scholze parameter $\varphi$ is generous. Decompose $V_\varphi$ in $\Rep(S_\varphi)^{BW_{E_\bu}}$ as $\bigoplus_{\delta\in \Irr(S_\varphi,\chi_b)} \delta \boxtimes V_{\varphi,\delta}$. Then we have an isomorphism
\[
R\Gamma(G,b,\mu_\bu)[\rho] \cong \bigoplus_{\delta\in \Irr(S_\varphi,\chi_b)} i^{1\ast}\Act_{\delta}(i^b_\ast \rho) \boxtimes V_{\varphi,\delta}
\]
in $\D^\varphi(G(E))^{BW_{E_\bu}}$, and each $i^{1\ast}\Act_{\delta}(i^b_\ast\rho)$ is a bounded complex of finite length $G(E)$-representations.
\end{theo}

\begin{proof}
The formula follows immediately from equations (\ref{Shtukas and Hecke}) and (\ref{generic factorization}), and the last statement follows from the fact that $i^{1\ast}\Act_{\delta}(i^b_\ast\rho)$ is a direct summand of $R\Gamma(G,b,\mu_\bu)[\rho]$.
\end{proof}

\begin{rema}\label{remarks on main}
We make some remarks on Theorem \ref{main formula}.

\begin{enumerate}

\item We expect that the theorem holds when $\varphi$ is strongly semisimple. For this, one would want to show that $i^b_\ast \rho$ has a $\QCoh([\ast/S_\varphi])$-structure, but this is stronger than $\rho$ having Fargues--Scholze parameter $\varphi$ when $\varphi$ is not generous. In situations when $\varphi$ lifts to a generous parameter on an isogenous group, we expect that one can prove this stronger statement, but we have not checked the details.

\smallskip

\item The categorical conjecture \cite[Conj. X.1.4]{fargues-scholze} predicts that $\D^\varphi_{\lis}(\Bun_G)$ is equivalent to $\D(\Rep(S_\varphi))$ when $\varphi$ is generous. In particular, each $i^{1\ast}\Act_{\delta}(i^b_\ast\rho)$ should be a split complex.

\smallskip

\item Each $\Act_{\delta}(i^b_\ast\rho)$ is non-zero but, unless $\varphi$ is elliptic, many $i^{1\ast}\Act_{\delta}(i^b_\ast\rho)$ will be zero (indeed, only finitely many can be non-zero).

\smallskip

\item Let us briefly compare our results to those of Koshikawa \cite{koshikawa-es}, in particular Thm. 1.3 of \emph{loc. cit}, which says that the $W_{E_\bu}$-representations appearing in $R\Gamma(G,b,\mu_\bu)[\rho]$ are subquotients of $V_\varphi$, without any assumption on $\varphi$ and allowing integral coefficents. In characteristic $0$, we can recover this statement from the assertion that $i^{1\ast} T_V i_{b}^\ast \rho \in \D^\varphi(G(E))^{BW_{E\bu}}$ (which does not require any assumption on $\varphi$). Moreover, a formula of the form  $R\Gamma(G,b,\mu_\bu)[\rho] \cong \bigoplus_{\delta\in \Irr(S_\varphi,\chi_b)} C_\delta \boxtimes V_{\varphi,\delta}$ can be deduced straight from the spectral action as in the proof of \cite[Thm. 1.3]{koshikawa-es}, without the machinery developed here. Indeed, the spectral action shows that the $W_{E_\bu}$-action on $R\Gamma(G,b,\mu_\bu)[\rho]$ factors as
\[
\qlb[W_{E_\bu}] \to \End(\mc{V})/\m_\varphi \to \End(R\Gamma(G,b,\mu_\bu)[\rho]),
\]
where $\m_\varphi \sub \ZZ^{\spec}(G)$ is the maximal ideal cutting out $\varphi$ and $\mc{V}$ is the $W_{E_\bu}$-equivariant vector bundle on $\X$ corresponding to $V$. When $\varphi$ is generous, one has $\End(\mc{V})/\m_\varphi = \End_{S_\varphi}(V_\varphi)$, and by standard representation theory one gets the desired decomposition. The extra information gained from factoring the spectral action is the functorial formula for the $C_\delta$.

\end{enumerate}
\end{rema}

In the rest of this section, we address points (2) and (3) of Remark \ref{remarks on main} when $\varphi$ is elliptic. Recall that $\varphi$ is said to be elliptic if it is semisimple and $S_\varphi/Z(\wh{G})^Q$ is finite. In this case, the discussion in \cite[\S X.2]{fargues-scholze} shows that, for any $b\in B(G)$, a $G_b(E)$-representation $\sigma$ with Fargues--Scholze parameter $\varphi$ has to be supercuspidal, and $b$ has to be basic. In particular,
\begin{equation}\label{splitting}
\prod_{b \in B(G)_{\basic}} i^{b\ast}: \D_{\lis}^\varphi(\Bun_G)  \longrightarrow \prod_{b \in B(G)_{\basic}} \D^\varphi(G_b(E))
\end{equation}
is an equivalence. We then have the following assertion.

\begin{lemm}\label{semisimplicity}
Assume that $\varphi$ is elliptic and that $b$ is basic. Then $\D^\varphi(G_b(E))$ is a product of categories of the form $\D(A)$, for semisimple (not necessarily commutative) Artinian rings $A$. In particular, any compact object in $\D^\varphi(G_b(E))$ is equivalent to a finite direct sum of shifted supercuspidal representations. Moreover, $\varphi$ is generous.
\end{lemm}

\begin{proof}
Let $C \sub X$ be the connected component containing $\varphi$, let $C^\Box$ be its preimage in $X^\Box$, and let $Y_{ur}$ the group variety of unramified characters $W_E \to Z(\wh{G}) \rtimes Q$. By the local Langlands correspondence for tori, $Y_{ur}$ is isomorphic to the group variety of unramified smooth characters of $G_b(E)$. As remarked just after \cite[Def X.1.2]{fargues-scholze}\footnote{The deformation-theoretic argument in \cite{fargues-scholze} appears to be incomplete. It shows that the family of unramified twists of $\varphi$ is an open subset of $C$, but one needs an argument to show that it is closed. If $Z(\wh{G})^Q$ is finite, i.e. the connected split center of $G$ is trivial, then there are no non-trivial unramified characters, hence the family is a point and therefore closed. In general, one can reduce to this case by considering the map from the moduli space of $L$-parameters of $G$ to the moduli space of $L$-parameters of the quotient of $G$ by its connected split center. The description of $C$ can also be extracted from \cite[Thm. 1.7]{dhkm}.}, $C$ consists of all twists of $\varphi$ by elements of $Y_{ur}$, and all these are strongly semisimple. In particular, $C^\Box \to C$ is smooth with fibres isomorphic to $\wh{G}/S_\varphi$, showing that $\varphi$ is generous.

\medskip

Now consider $\D^C_{\lis}(\Bun_G)$, which is equivalent to the product of the $\D^C(G_b(E))$ for $b\in B(G)_{\basic}$. Each $\D^C(G_b(E))$ is the product of its Bernstein components, all of which are supercuspidal, and the action of $\QCoh(C)$ preserves the Bernstein components. Let $\mc{C}$ be such a Bernstein component, let $\ZZ(\mc{C})$ denote its center and let $\mc{R}(\mc{C})$ denote the endomorphism ring of a compact generator of $\mc{C}$. We have that $\QCoh(C)$ is equivalent to $\D(\oo(C))$ and $\mc{C}$ is equivalent to $\D(\mc{R}(\mc{C}))$. Thus, by Proposition \ref{actions}, the action of $\QCoh(C)$ on $\mc{C}$, viewed as an action of $\D(\oo(C))$ on $\D(\mc{R}(\mc{C}))$, is given by letting $M \in \D(\oo(C))$ act by the endomorphism
\[
N \mapsto M \otimes_{\oo(C),f}N
\]
on $\D(\mc{R}(\mc{C}))$, where $f : \oo(C) \to \ZZ(\mc{C})$ is the induced homomorphism on centres. Both $\oo(C)$ and $\ZZ(\mc{C})$ carry twisting actions of $Y_{ur}$; indeed by choosing a base point they are both isomorphic to quotients of $Y_{ur}$ by finite groups. Since the Fargues--Scholze construction is compatible with twisting \cite[Thm. IX.0.5(ii)]{fargues-scholze}, $f$ is equivariant for the actions of $Y_{ur}$, and hence finite Galois.

\medskip

Armed with this, we now consider $\D^\varphi(G_b(E))$, which is the product of the categories $\mc{C}^\varphi := \Mod_{\oo(C)/\m_\varphi}\mc{C}$, where $\mc{C}$ ranges over the Bernstein components of $\D^C(G_b(E))$ and $\m_\varphi \sub \oo(C)$ is the maximal ideal corresponding to $\varphi$. By the description of the action of $\QCoh(C)$ on $\mc{C}$, $\mc{C}^\varphi$ is equivalent to $\D(\oo(C)/\m_\varphi \otimes^L_{\oo(C),f}\mc{R}(\mc{C}))$, so we need to show that $\oo(C)/\m_\varphi \otimes^L_{\oo(C),f}\mc{R}(\mc{C})$ is a (classical) semisimple Artinian ring. By \cite[8.1, Prop.]{bushnell-henniart}, $\mc{R}(\mc{C})$ is an Azumaya algebra over $\ZZ(\mc{C})$. Since $f$ is finite Galois, it follows that $\oo(C)/\m_\varphi \otimes^L_{\oo(C),f}\mc{R}(\mc{C})$ is concentrated in degree $0$, where it is a finite product of matrix algebras over $\qlb$, and in particular semisimple and Artinian, as desired. The rest of the lemma then follows immediately. 
\end{proof}

We now get the following refinement of Theorem \ref{main formula} for elliptic parameters.

\begin{coro}\label{main elliptic}
In the situation of Theorem \ref{main formula}, assume additionally that $\varphi$ is elliptic. Then each $i^{1\ast}\Act_{\delta}(i^b_\ast\rho)$ is a non-zero finite direct sum of shifted supercuspidal representations. If $\delta$ is one-dimensional, then $i^{1\ast}\Act_{\delta}(i^b_\ast\rho)$ is a single supercuspidal representation, up to shift.
\end{coro}

\begin{proof}
$\Act_{\delta}(i^b_\ast\rho)$ is non-zero and supported on $\Bun_G^1$ by equation (\ref{splitting}), hence $i^{1\ast}\Act_{\delta}(i^b_\ast\rho)$ is non-zero. It is a finite direct sum of supercuspidal representations up to shift by Lemma \ref{semisimplicity}. Finally, if $\delta$ is one-dimensional, then $\Act_{\delta}$ is an equivalence, so $\End(i^{1\ast}\Act_{\delta}(i^b_\ast\rho)) = \End(\rho) = \qlb$ by the previous observation on the support. It follows that $i^{1\ast}\Act_{\delta}(i^b_\ast\rho)$ is a single supercuspidal representation up to shift.
\end{proof}

\begin{rema}

We note that Corollary \ref{main elliptic} has previously appeared in the literature; see \cite[\S X.2]{fargues-scholze} for the case when $S_\varphi$ is finite and \cite[Cor. 3.11]{hamann}, \cite[Thm. 2.27]{bmhn} for the general case, where it plays a key role in comparing the Fargues--Scholze construction with the local Langlands correspondence. Note that, in light of Corollary \ref{main elliptic}, the Kottwitz conjecture (with respect to the Fargues--Scholze construction) amounts to computing the image of $i^{1\ast}\Act_{\delta}(i^b_\ast \rho)$ in the Grothendieck group of $G(E)$. The vanishing conjecture, on the other hand, says that $i_1^\ast\Act_{\delta}(i_{b\ast}\rho)$ is concentrated in degree $0$. In particular, Corollary \ref{main elliptic} gives a complete understanding of the action of the Weil group. This is in contrast to \emph{local} approaches to the cohomology of local Shimura varieties predating \cite{fargues-scholze}, which could say very little about the Weil group action.
\end{rema}

\section{Some consequences}\label{sec: consequences}

In this section we indicate some results towards Conjecture \ref{kottwitz conjecture} that can be obtained from Theorem \ref{main formula} using simple tricks or observations, or recent works on Kottwitz conjecture. We will use the same notation as in \S \ref{sec: formula}, with a few extra additions. For clarity, we gather this here. These assumptions will be in place throughout the whole of \S \ref{sec: consequences}.

\begin{setup}\label{setup sec 4}
We let $\rho$ be an irreducible admissible representation of $G_b(E)$. We will make the following assumptions, and use the following notation:

\begin{itemize}
\item The Fargues--Scholze parameter $\varphi$ of $\rho$ is elliptic (this forces $\rho$ to be supercuspidal, as noted above).

\smallskip

\item Since $\varphi$ is elliptic, the set $\Irr(S_\varphi, \chi_b)$ is finite. We let $\delta_1,\dots,\delta_r$ denote its members, where $r=\# \Irr(S_\varphi, \chi_b)$.

\smallskip

\item With $V_\varphi$ as in Theorem \ref{main formula}, write $V_{\varphi,i} = \Hom_{S_\varphi}(\delta_i,V_\varphi)$, viewed as a $W_{E_\bu}$-representation.

\smallskip

\item Finally, assume that $E$ has characteristic $0$. 
\end{itemize}
\end{setup}

\subsection{The Kottwitz conjecture} Let us start with the simplest possible version of Corollary \ref{main elliptic}: Assume that $\varphi$ is stable, i.e. $S_\varphi = Z(\wh{G})^Q$. Then the formula reads
\[
R\Gamma(G,b,\mu_\bu)[\rho] \cong i^{1\ast}\Act_{\chi_b}(i^b_\ast\rho) \boxtimes V_\varphi,
\]
and $i^{1\ast}\Act_{\chi_b}(i^b_\ast\rho)$ is a single supercuspidal representation up to shift. The following proposition is then clear (note that $\chi_1$ is the trivial representation).

\begin{prop}
Assume that $\varphi$ is stable. Then $R\Gamma(G,b,\mu_\bu)[\rho]$ vanishes outside a single degree, and in that degree it is the exterior tensor product of a supercuspidal representation and $V_\varphi$. If, additionally, $b=1$, then $R\Gamma(G,b,\mu_\bu)[\rho] = \rho \boxtimes V_\varphi$, i.e. the strong Kottwitz conjecture holds.
\end{prop}

We now return to the general case. In \cite[Thm 1.0.2]{hkw}, the Kottwitz conjecture is proven after disregarding the action of $W_{E_\bu}$, under further assumptions on $L$-packets. We wish to combine this with Theorem \ref{main formula} in order to deduce Corollary \ref{intro kottwitz} from the introduction. This will use the following assumptions and notation:

\begin{setup}\label{setup local Langlands}
Assume the refined local Langlands conjecture in the form of \cite[Conjecture G]{kaletha}. In particular, we have an $L$-parameter $\phi$ attached to $\rho$, and an $L$-packet $\Pi_\phi(G)$ (a set) of $G(E)$-representations. We assume that $\phi$ is discrete (i.e. that $S_\phi/Z(\wh{G})^Q$ is finite), with centralizer $S_\phi$, and that all members of $\Pi_\phi(G)$ are supercuspidal.
\end{setup}

Under these assumptions, \cite[Thm. 1.0.2]{hkw} asserts that
\[
\Mant_{G,b,\mu_\bu}(\rho) = \sum_{\pi \in \Pi_\phi(G)} \dim \Hom_{S_\phi}(\delta_{\pi,\rho},V_\phi) \cdot \pi 
\] 
in the Grothendieck group of $G(E)$-representations (recall from the introduction that $\Mant_{G,b,\mu_\bu}(\rho):= \sum_{n} (-1)^n H^n(R\Gamma(G,b,\mu_\bu)[\rho]$). Here $\delta_{\pi,\rho}$ is a certain (not necessarily irreducible) representation of $S_\phi$ attached to the pair $(\pi,\rho)$ with $Z(\wh{G})^Q$ acting by $\chi_b$, described in \cite[\S 2.3]{hkw}. We will need the following lemma concerning the representations $\delta_{\pi,\rho}$.

\begin{lemm}\label{multiplicities}
Any $\delta\in \Irr(S_\phi,\chi_b)$ occurs as a subrepresentation of  $\bigoplus_{\pi\in \Pi_\phi(G)}\delta_{\pi,\rho}$. If all $\delta\in \Irr(S_\phi,\chi_b)$ are one-dimensional, then we have
\[
\sum_{\pi\in \Pi_\phi(G)} \dim \Hom_{S_\phi}(\delta_{\pi,\rho},W) =  \dim \Hom_{S_\phi}(\oplus_{\pi\in \Pi_\phi(G)}\delta_{\pi,\rho},W) \geq \dim W
\]
for any $W \in \Rep(S_\phi)$ with $Z(\wh{G})^Q$ acting as $\chi_b$.
\end{lemm} 

\begin{proof}
For this lemma, we will refer to \cite[\S 2.3]{hkw} and the references within for unexplained terminology, notation, and facts (see also \cite[\S 4.5]{kaletha}). There is a group $S_\phi^+$ arising as a covering group $f : S_\phi^+ \to S_\phi$. Set $Z(\wh{\bar{G}})^+ := f^{-1}(Z(\wh{G})^Q)$; this group is central in $S_\phi^+$. It induces a map $\pi_0(Z(\wh{\bar{G}})) \to \pi_0(S_\phi^+)$. If $\lambda$ is a character of $\pi_0(Z(\wh{\bar{G}}))$, we will say that a representation of $\pi_0(S_\phi^+)$ has character $\lambda$ if $\pi_0(Z(\wh{\bar{G}}))$ acts as $\lambda$ on it. According to \cite[Conjecture G]{kaletha}, there is a bijection $\pi \to \tau_{z,\mf{m},\pi}$ between $\Pi_\phi(G)$ and irreducible representations of $\pi_0(S_\phi^+)$ with a certain character $\lambda_z$, and there is an irreducible representation $\tau_{z,\mf{m},\rho}$ of $\pi_0(S_\varphi^+)$ associated with $\rho$, which has a certain character $\lambda_z + \lambda_{z_b}$. By definition, $\delta_{\pi,\rho}$ is the descent of $\tau_{z,\mf{m},\pi}^\ast \otimes \tau_{z,\mf{m},\rho}$ to $S_\phi$ along $f$ (and hence $\lambda_{z_b}$ corresponds to $\chi_b$). With this in mind, let $\delta$ be an irreducible representation of $S_\phi$ with character $\chi_b$. Pullback along $f$ gives an equivalence of representations of $S_\phi$ with character $\chi_b$ and representations of $\pi_=(S_\phi^+)$ with character $\lambda_{z_b}$, so we also let $\delta$ denote the corresponding representation of $\pi_0(S_\phi^+)$. Then $\delta^\ast \otimes \tau_{z,\mf{m},\rho}$ has character $\lambda_z$, and can therefore be written as a sum of the $\tau_{z,\mf{m},\pi}$ (possibly with multiplicities). Moreover, $\delta$ is a summand of $(\delta^\ast \otimes \tau_{z,\mf{m},\rho})^\ast \otimes \tau_{z,\mf{m},\rho}$. This proves the first part of the lemma. To deduce the second part, we now have $\oplus_{\delta\in\Irr(S_\phi,\chi_b)}\delta  \sub \oplus_{\pi\in \Pi_\phi(G)}\delta_{\pi,\rho}$ (non-canonically), so
\[
\dim \Hom_{S_\phi}(\oplus_{\pi\in \Pi_\phi(G)}\delta_{\pi,\rho},W) \geq \dim \Hom_{S_\phi}(\oplus_{\delta\in\Irr(S_\phi,\chi_b)}\delta,W) \geq \dim W,
\]
where the second inequality uses that the $\delta$ are one-dimensional.
\end{proof}

With these preparations, we can deduce Corollary \ref{intro kottwitz}.

\begin{theo}\label{Kottwitz}
Assume Setups \ref{setup sec 4} and \ref{setup local Langlands}, and that all $\delta_i$ and all $\delta \in \Irr(S_\phi,\chi_b)$ are one-dimensional. Write $i^{1\ast}\Act_{\delta_i}(i^b_\ast\rho) = \pi_i^\prime[n_i] $, where $\pi_i^\prime$ is supercuspidal and $-[n_i]$ denotes a shift by an integer $n_i$ (we can do this by Corollary \ref{main elliptic}). Then
\[
\Mant_{G,b,\mu_\bu}(\rho) = \sum_{i=1}^r \pi_1^\prime \boxtimes V_{\varphi,i},
\]
and $\{ \pi_1^\prime,\dots,\pi_r^\prime\} = \Pi_\phi(G)$ as sets. Moreover, $H^n(R\Gamma(G,b,\mu_\bu)[\rho])=0$ when $n$ is odd.
\end{theo}

\begin{proof}
By Theorem \ref{main formula}, $\Mant_{G,b,\mu_\bu}(\rho) = \sum_{i=1}^r (-1)^{n_i} \pi_i^\prime \boxtimes V_{\varphi,i}$ in $\mathrm{Groth}(G(E)\times W_{E_\bu})$. In particular, $\Mant_{G,b,\mu_\bu}(\rho) = \sum_{i=1}^r (-1)^{n_i} \dim V_{\varphi,i} \cdot \pi_i^\prime $ in $\mathrm{Groth}(G(E))$, where we have that
\[
\sum_{i=1}^r (-1)^{n_i} \dim V_{\varphi,i} \leq \sum_{i=1}^r  \dim V_{\varphi,i} = \dim V
\]
since the $\delta_i$ run through $\Irr(S_\varphi,\chi_b)$. On the other hand,
\[
\Mant_{G,b,\mu_\bu}(\rho) = \sum_{\pi\in \Pi_\phi(G)}  \dim \Hom_{S_\phi}(\delta_{\pi,\rho},V_\phi) \cdot \pi
\]
in $\mathrm{Groth}(G(E))$ by \cite[Thm. 1.0.2]{hkw}. By Lemma \ref{multiplicities}, we have $\sum_{\pi\in \Pi_\phi(G)}  \dim \Hom_{S_\phi}(\delta_{\pi,\rho},V_\phi) \geq \dim V$. Note that, for a fixed $i$ (or $\pi$) and varying $\mu_\bu$, some $V_{\varphi,i}$ (or $\Hom_{S_\phi}(\delta_{\pi,\rho},V_\phi)$) will be non-zero. Using this fact and comparing the two inequalities we have derived so far, we see that the $n_i$ are even and that $\{ \pi_1^\prime,\dots,\pi_r^\prime \} = \Pi_\varphi(G)$, as desired.
\end{proof}

\begin{rema}
We make a few remarks on this theorem.

\begin{enumerate}

\item While \cite[Thm. 1.0.2]{hkw} is only stated for $\mu_\bu = \mu$ a single cocharacter, its proof goes through for general $\mu_\bu$. We note, however, that Theorem \ref{Kottwitz} for general $\mu_\bu$ follows from the case $\mu_\bu = \mu$ (in light of Corollary \ref{main elliptic}), as this case suffices to show that $\{ \pi_1^\prime,\dots,\pi_r^\prime\} = \Pi_\phi(G)$ and that the $n_i$ are even.

\smallskip

\item As noted in the introduction, this is rather close to the Kottwitz conjecture, but falls short in two aspects. Firstly, we do not know that $\phi = \varphi$ (indeed, we do not known that $\varphi$ is elliptic in most cases it is expected to be). Lacking this, our result looks rather amusing: The $G(E)$-representations in $R\Gamma(G,b,\mu_\bu)[\rho]$ arise from $\phi$, whereas the $W_{E_\bu}$-action is given in terms of $\varphi$. Moreover, even if $\phi = \varphi$, it is not clear to us how the resulting parametrizations of $\Pi_\phi(G)$ compare. At present, the only approach available to these questions is through the cohomology of (global) Shimura varieties. It is known that $\phi = \varphi$ for inner forms of $\GL_n$ and $\SL_n$ for general $E$ \cite{fargues-scholze,hkw}, for inner forms of $\GSp_4$ and $\Sp_4$ when $E$ is unramified over $\Qp$ \cite{hamann}, for some (similitude) unitary groups in an odd number of variables \cite{bmhn} when $E=\Qp$, and for special orthogonal and unramified unitary groups when $E$ is unramified \cite{peng}. Comparing the parametrizations of the $L$-packet seems to be more subtle, however.

\smallskip

\item One-dimensionality is known to hold in many cases. Examples include inner forms of $\SL_n$, $\Sp_{2n}$, $\SO_{2n+1}$ and unitary groups. Moreover, \cite[Conjecture G]{kaletha} is known for the regular supercuspidal $L$-packets constructed by Kaletha for all $G$ split over a tame extension of $E$ and for $p$ sufficiently large, in \cite{kaletha-regular}; see \cite{kaletha-regular,fks}. We refer to the introduction of \cite{hkw} for more details.
\end{enumerate}

\end{rema}

\subsection{Vanishing}

Recall that we assume Setup \ref{setup sec 4} throughout \S \ref{sec: consequences}. One interesting aspect of Theorem \ref{main formula} is its uniformity when varying $\mu_\bu$: Only the $W_{E_\bu}$-representations change\footnote{A trivial consequence is that, for fixed $(G,b)$ and $\rho$, there exists an $N$ such that $H^i(R\Gamma(G,b,\mu_\bu)[\rho])$ vanishes for all $|i|>N$ and all $\mu_\bu$.}. This observation can sometimes be used to propagate vanishing results for some $\mu_\bu$ to a larger collection\footnote{Conjecture \ref{kottwitz conjecture}(2) was recently proven by the first author in \cite{hansen} in many cases where $\mu_\bu$ is a single minusucle cocharacter.}. To illustrate the method, we reprove the strong Kottwitz conjecture for inner forms of $\GL_n$, which was previously proved by the first author \cite[Thm. 1.9]{hansen}, using results on averaging functors from \cite{anschutz-lebras}.

\begin{theo}\label{strong Kottwitz GLn}
Assume Setup \ref{setup sec 4} and that $G$ is an inner form of $\GL_n$. Then $R\Gamma(G,b,\mu_\bu)[\rho] \cong \pi \boxtimes V_\varphi$, where the right hand is concentrated in degree $0$ and $\pi$ is the irreducible admissible representation of $G(E)$ corresponding to $\varphi$ under the usual local Langlands correspondence. 
\end{theo}

\begin{proof}
Note that we are in the stable situation here, so we need to show that $i^{1\ast}\Act_{\chi_b}(i^b_\ast\rho)$ is $\pi$ concentrated in degree $0$. That it is a shift of $\pi$ follows from \cite[Thm. 1.0.3]{hkw}, since its Fargues--Scholze $L$-parameter is $\varphi$. To prove vanishing, note that we can choose a minuscule cocharacter $\mu_b$ such that $R\Gamma(G,b,\mu_b)[\rho] = i^{1\ast}\Act_{\chi_b}(i^b_\ast\rho) \boxtimes W_\varphi$, where $W$ the dual of the irreducible representation with extreme weight $\mu_b$. By \cite[Thm. 1.6]{hansen}, $R\Gamma(G,b,\mu_b)[\rho]$ vanishes outside degree $0$, finishing the proof.
\end{proof}

Another trick that can sometimes be used is duality. Letting $\bD$ denote Verdier duality on $\D_{\lis}(\Bun_G)$, one has
\[
\bD(R\Gamma(G,b,\mu_\bu)[\rho]) \cong \bD(i^{1\ast} T_V i^b_\ast \rho) \cong i^{1\ast} T_V i^b_\ast \rho^\vee \cong R\Gamma(G,b,\mu_\bu)[\rho^\vee],
\]
as objects in $\D_{\lis}(\Bun_G)$ (i.e. forgetting the $W_{E_\bu}$-action), where we have used that $i_1^\ast = i_1^!$, $i_{b!}\rho^\vee = i_{b\ast}\rho^\vee$ (by \cite[Thm 1.3]{hansen} and \cite[Thm IX.05(iv)]{fargues-scholze}, since $\varphi$ is elliptic), and the interplay between $\bD$ and pullback/pushforward, and the Hecke operators \cite[Thm. IX.0.1(i)]{fargues-scholze}. This gives us the following vanishing theorem. In its formulation, we note that there is a canonical isomorphism between the cocenters of $G$ and $G_b$ over $E$, so any smooth character of $G(E)$ can be naturally viewed as a smooth character of $G_b(E)$ (and vice versa).

\begin{prop}\label{vanishing}
Assume Setup \ref{setup sec 4}, that the $\delta_i$ are one-dimensional, and that, writing $i^{1\ast} \Act_{\delta_i}(i^b_\ast\rho) = \pi_i [n_i]$, the $\pi_i$ are distinct. Assume further that there is a smooth character $\chi$ of $G(E)$ such that $\rho^\vee \cong \rho \otimes \chi$ and $\pi_i^\vee \cong \pi_i \otimes \chi$ for all $i$. Then $R\Gamma(G,b,\mu_\bu)[\rho]$ is concentrated in degree $0$.
\end{prop}

\begin{proof}
Since $\rho^\vee \cong \rho \otimes \chi$, Verdier duality gives us that $\bD(R\Gamma(G,b,\mu_\bu)[\rho]) \cong R\Gamma(G,b,\mu_\bu)[\rho \otimes \chi]$. Computing from the definitions, one sees that $R\Gamma(G,b,\mu_\bu)[\rho \otimes \chi] \cong R\Gamma(G,b,\mu_\bu)[\rho] \otimes \chi$. Thus, evaluating both sides as a $G(E)$-representation using Theorem \ref{main formula} and Corollary \ref{main elliptic}, we get that
\[
\bigoplus_{i=1}^r \pi_i^\vee[-n_i]^{\oplus \dim V_{\varphi,i}} \cong \bigoplus_{i=1}^r (\pi_i \otimes \chi)[n_i]^{\oplus \dim V_{\varphi,i}}.
\]
Since $\pi^\vee_i \cong \pi_i \otimes \chi$ and the $\pi_i$ are assumed to be distinct, it follows that $n_i=0$ for all $i$ with non-zero multiplicity, as desired.
\end{proof}

\begin{rema}\label{vanishing remark}
Let us finish with a few remarks on Proposition \ref{vanishing}, and the vanishing conjecture more generally.

\begin{enumerate}
\item If $\varphi$ is stable, the assumption $\pi_i^\vee \cong \pi_i \otimes \chi$ can be dropped, and instead deduced as a consequence of the argument.

\smallskip

\item If, in addition to the assumptions in Proposition \ref{vanishing}, we also assume Setup \ref{setup local Langlands}, then the assertion that the $\pi_i$ are distinct is equivalent to $\Pi_\phi(G)$ having size $r$. For example, this is true if $S_\phi/Z(\wh{G})^Q \cong S_\varphi / Z(\wh{G})^Q$ and $G$ is quasisplit (then $\Pi_\phi(G)$ is in bijection with $\Irr(S_\varphi/Z(\wh{G})^Q)$, which is equal to $\{\delta_i \delta_j^{-1} \mid i=1,\dots,r\}$ for arbitrary $j$). On the other hand, if $G$ is the non-split inner form of $\SL_2$ and $S_\varphi = (\Z /2)^2$, then we know that $\phi=\varphi$ and that $\Pi_\phi(G)$ has size $1$.

\smallskip

\item Essential self-duality is a fairly common feature among representations of classical groups and related groups. For example, it holds for all representations of $\GSp_{2g}$; see \cite[Remark 5]{prasad} (in this case $\chi$ is determined by the central character and hence by $\varphi$). Other examples include $\SO_{2n+1}$ and the unique non-split inner form of $\GSp_4$.

\smallskip

\item Conjecture \ref{intro kottwitz}(2) can sometimes be passed through isogenies. For example, one can deduce it for $\SL_n$ and $\Sp_4$ from the case of $\GL_n$ and $\GSp_4$, respectively, using computations as in the proof of \cite[Thm. IX.6.1]{fargues-scholze}.
\end{enumerate}
\end{rema}

\bibliographystyle{alpha}
\bibliography{kottwitzbib}

\end{document}